\documentclass[smallextended]{svjour3}       
\smartqed  
\usepackage{graphicx}
\usepackage{amssymb}
\usepackage{epstopdf}
\usepackage[caption=false]{subfig}
\usepackage{amssymb}
\usepackage{hyperref}
\usepackage{graphicx}
\usepackage{stfloats}
\usepackage{algorithm}
\usepackage{algorithmic}
\usepackage{enumerate}
\usepackage{amsmath}
\usepackage{tabularx}
\usepackage{multirow}
\usepackage{booktabs}
\usepackage{setspace}
\usepackage[numbers,sort&compress]{natbib}
\usepackage[top=2cm, bottom=2cm, left=2.4cm, right=2.8cm]{geometry}
\begin{document}

\setlength{\baselineskip}{18pt}

\title{An explicit extragradient algorithm for equilibrium problems on Hadamard manifolds
}


\author{Jingjing Fan$^*$  \and Bing Tan \and Songxiao Li
}


\institute{$^*$Corresponding author. \at
              Institute of Fundamental and Frontier Sciences, University of Electronic Science and Technology of China, Chengdu 611731, China\\
                       \email{fanjingjing0324@163.com} (J. Fan),{bingtan72@gmail.com} (B. Tan),   {jyulsx@163.com} (S. Li)
}

\date{Received: date / Accepted: date}

\maketitle

\begin{abstract}
In this paper, we investigate a new extragradient algorithm for solving pseudomonotone equilibrium problems on Hadamard manifolds.
The algorithm uses a variable stepsize which is updated at each iteration and based on some previous iterates.
The convergence analysis of the proposed algorithm is discussed under mild assumptions.
In the case where the equilibrium bifunction is strongly pseudomonotone, the $R$-linear rate of convergence of the new algorithm is formulated.
A fundamental experiment is provided to illustrate the numerical behavior of the algorithm.
The results presented in this paper generalize some corresponding known results.
\end{abstract}

\section{Introduction}

Let $C$ be a nonempty  convex and closed subset of a real Hilbert space $H$.
Let $f: C \times C \rightarrow \mathbb{R}$ be a bifunction with $f (x, x) = 0$ for all $x \in C$.
Consider the problem involving $f$, which consists of finding $x^* \in  C$ such that
$
f(x^*,y) \geq 0,$ $ \forall y\in C.
$
The   problem, which  is also called the Ky Fan inequality, was introduced by Fan \cite{Fan1}  and further   developed by Blum and Oettli \cite{Blu1}. It is now known and called the equilibrium problem.
The set of all solutions of the equilibrium problem is denoted by $EP(f, C)$.
Many problems arising in  transportation, financial engineering, and medical imaging  can be reduced to finding solutions of the equilibrium problems
see, for example, \cite{Kon1,Fac1,Ius2} and the references therein.

Recently, many numerical algorithms have been proposed for solving equilibrium problems such as the proximal point algorithm \cite{Chen1},
the extragradient algorithm \cite{Qin2}, the subgradient algorithm \cite{Bur1}, and the gap function algorithm  \cite{Mas1}.

In 2019, Hieu, Quy and Vy \cite{Hie2} introduced an extragradient algorithm to solve a pseudomonotone equilibrium problem  with a Lipschitz-type condition
in $H$. The extragradient algorithm is as following. Given $x_0 \in  C$ and $\lambda_0 >0, \mu \in (0,1)$, compute $y_n$ and $x_{n+1}$ by
\[
\begin{aligned}
\left\{
     \begin{array}{ll}
        y_n=\arg\min_{y\in C} \left\{f(x_n,y)+\frac{1}{2\lambda_n} \|x_n-y\|^2 \right\} ,\\
        x_{n+1}=\arg\min_{y\in C} \left\{f(y_n,y)+\frac{1}{2\lambda_n} \|x_n-y\|^2 \right\},
        \end{array}
    \right.
\end{aligned}
\]
where
$$
\lambda_{n+1}=\min \left\{ \lambda_n, \frac{\mu(\|x_n-y_n\|^2+\|x_{n+1}-y_n\|^2)}{2[f(x_n,x_{n+1})-f(x_n,y_n)-f(y_n,x_{n+1})]_+} \right\}.
$$
They proved that   iterative scheme $\{x_n\}$ converges to some $x^* \in EP(f, C)$.

On the other hand,   in many practical applications, the natural structure of the data can be  modeled as constrained optimization problems,
where the constraints are non-linear and non-convex. More specially, the constraints are Riemannian manifolds, see, e.g., \cite{Bac1,Ber2}.
Many issues in nonlinear analysis such as fixed point problems, and variational inequalities have been magnified from linear settings to nonlinear systems because these problems cannot be posted in   linear spaces and require a manifold structure.
Therefore, the extensions of the concepts and techniques in    equilibrium problems and related topics from Euclidian spaces to Riemannian manifolds are natural,
and the generalizations of optimization methods from Euclidean spaces to Riemannian manifolds also have some more important advantages,
see, for example, \cite{Lib1,Ded3,Li4,ansari5}.

In 2012, Colao et al. \cite{Col1} on the  Riemannian setting first introduced the equilibrium problems, which consists of finding $x^* \in  C$ such that
\begin{equation*}\label{EP}
f\left(x^{*}, y\right) \geq 0, \quad \forall y \in C, \tag{\text{EP}}
\end{equation*}
where $C$ is a nonempty  convex and closed subset of Hadamard manifold $\mathcal{M}$, and
$f : C \times C\rightarrow \mathbb{R}$ is a bifunction with $f(x, x)=0$, for all $x \in C$.
We denote by $EP(f, C)$ the set solution of  problem \eqref{EP}.
Indeed, in recent years, various algorithms, which involves monotone bifunctions,
have been extended
 to solve equilibrium problems from Hilbert spaces to the more general setting of Riemannian manifolds.
In particular, Khammahawong et al. \cite{Kha1} presented an extragradient algorithm to solve strongly pseudomonotone equilibrium problems on Hadamard manifolds.
Their algorithm is described as follows. Given $x_0, y_0 \in C$,  compute $x_{n+1}$ and $y_{n+1}$ by
\[
\begin{aligned}
\left\{
     \begin{array}{ll}
        x_{n+1}=\arg\min_{y\in C} \left\{f(y_n,y)+\frac{1}{2\lambda_n}d^2(x_n,y) \right\} ,\\
        y_{n+1}=\arg\min_{y\in C} \left\{f(y_n,y)+\frac{1}{2\lambda_{n+1}}d^2(x_{n+1},y) \right\},
        \end{array}
    \right.
\end{aligned}
\]
where $d(x, y)$ is the geodesic distance between $x$ and $y$ in $\mathcal{M}$
and nonincreasing sequence $\{\lambda_n\}$ satisfying $\lim_{n\rightarrow \infty}\lambda_n=0$ and $\sum^{\infty}_{n=0}\lambda_n=+\infty$.
The convergence of sequence $\{x_n\}$ was investigated and obtained.

Inspired by the work in \cite{Hie2,Kha1}, the aim of this paper is to present an extragradient algorithm with new stepsize rules
for pseudomonotone equilibrium problems on Hadamard manifolds and   study its convergence properties.
Our algorithm uses a variable stepsize sequence, which is generated at each iteration, based on some previous iterates,
and without any linesearch procedure.
This leads to the main advantage of the algorithm, that is, the performance of its convergence is done without the prior
knowledge of the Lipschitz-type constants of bifunctions.
The convergence of the resulting algorithm is established under suitable conditions.
In the case that  the bifunction is strongly pseudomonotone, the $R$-linear rate of the convergence of the algorithm is also  proved.

The rest of this paper is organized as follows.
In Section 2, we present some basic definitions and fundamental results from manifolds which will be needed in the sequel.
In Section 3, we propose the new extragradient algorithm involving pseudomonotone bifunctions and analyze its convergence
on Hadamard manifolds.
In Section 4, we study the convergence rate of the proposed algorithm.
In Section 5, we give numerical experiments to illustrate the computational performance on a test problem.
Finally, Section 6, the last section,  concludes the paper with a brief summary.

\section{Preliminaries}

In this section, we recall some fundamental definitions, properties, and notations concerned with the Riemannian geometry.
These basic facts can be found, for example, in  \cite{Fer3,Lis1,Led1}.

Let $\mathcal{M}$ be a finite dimensional differentiable manifold. The set of all tangents at $x \in \mathcal{M}$ is called a tangent space of $\mathcal{M}$ at $x \in \mathcal{M}$, which forms a vector space of the same dimension as $\mathcal{M}$. And we denote it by $T_x\mathcal{M}$. The tangent bundle of $\mathcal{M}$ is denoted by $T\mathcal{M}=\bigcup_{x\in \mathcal{M}}T_x\mathcal{M}$, which is naturally a manifold. We denote by $\langle \cdot,\cdot \rangle_x$ the scalar product on $T_x\mathcal{M}$ with the associated norm $\|\cdot\|_x$, where the subscript $x$ is sometimes omitted. A differentiable manifold $\mathcal{M}$ with a Riemannian metric $\langle \cdot,\cdot\rangle$ is called a Riemannian manifold.
Letting $\gamma: [a,b] \rightarrow \mathcal{M}$ be a piecewise differentiable curve joining $x = \gamma(a)$ to $y = \gamma(b)$ in $\mathcal{M}$, we can define the length of $L(\gamma) =\int^b_a\|\gamma'(t)\|dt$. The minimal length of all such curves joining $x$ to $y$ is called the Riemannian distance and it is denoted by $d(x, y)$.

Let $\nabla$ be the Levi-Civita connection associated with the Riemannian metric. Let $\gamma$ be a smooth curve in $\mathcal{M}$. A vector field $X$ is said to be parallel along $\gamma$ iff $\nabla_{\gamma'}X=0$. If $\gamma'$ is parallel along $\gamma$, i.e., $\nabla_{\gamma'}\gamma' = 0$, then $\gamma$ is said to be geodesic. In this case, $\|\gamma'\|$ is a constant. Furthermore, if $\|\gamma'\|=1$, then $\gamma$ is called normalized.
A geodesic joining $x$ to $y$ in $\mathcal{M}$ is said to be minimal if its length equals $d(x, y)$. Let $\gamma: \mathbb{R} \rightarrow \mathcal{M}$ be a geodesic and $P_\gamma[.,.]$ denote the parallel transport along $\gamma$ with respect to $V$, which is defined by $P_{\gamma[\gamma(a),\gamma(b)]}(v) = V(\gamma(b))$ for all $a, b \in \mathbb{R}$ and $v \in T_{\gamma(a)}\mathcal{M}$, where $V$ is the unique
vector field satisfying $\nabla_{\gamma'(t)}V=0$ and $V(\gamma(a))=v$. Then, for any $a,b \in \mathbb{R}$, $P_{\gamma,[\gamma(b),\gamma(a)]}$ is an isometry from $T_{\gamma(a)}\mathcal{M}$ to $T_{\gamma(b)}\mathcal{M}$. We will write $P_{y,x}$ instead of $P_{\gamma,[y,x]}$ in the case where $\gamma$ is a minimal geodesic joining $x$ to $y$ if this will avoid any confusion.

A Riemannian manifold is complete if, for any $x \in \mathcal{M}$, all geodesics emanating from $x$ are defined for all $-\infty<t<+\infty$. By the Hopf-Rinow Theorem \cite{Sak1}, we know that if $\mathcal{M}$ is complete, then any pair of points in $\mathcal{M}$ can be joined by a minimal geodesic. Moreover, $(\mathcal{M}, d)$ is a complete metric space and bounded closed subsets are compact. If $\mathcal{M}$ is a complete Riemannian manifold, then the exponential map $\exp_x : T_x\mathcal{M} \rightarrow \mathcal{M}$ at $x$ is defined by $\exp_xv =\gamma_v(1, x)$ for each $v \in T_x\mathcal{M}$, where $\gamma(\cdot) = \gamma_v(\cdot,x)$ is the geodesic starting at $x$ with velocity $v$, that is, $\gamma(0)=x$ and $\gamma'(0)=v$. It is easy to see that $\exp_xtv = \gamma_v(t, x)$ for each real number $t$. Note that the mapping $\exp_x$ is differentiable on $T_x\mathcal{M}$ for any $x \in \mathcal{M}$. By the inverse mapping theorem, there exists an inverse exponential map $\exp^{-1}_x: \mathcal{M} \rightarrow T_x \mathcal{M}$. Moreover, the geodesic is the unique shortest path with $\|\exp^{-1}_xy\|= \|\exp^{-1}_yx\|= d(x, y)$, where $d(x, y)$ is the geodesic distance between $x$ and $y$ in $\mathcal{M}$. For further details, we refer to \cite{Sak1}.

A complete simply connected Riemannian manifold of nonpositive sectional curvature is called a Hadamard manifold. If $\mathcal{M}$ is a Hadamard manifold, then $\exp^{-1}_x: \mathcal{M} \rightarrow T_x \mathcal{M}$ is a diffeomorphism for every $x \in  \mathcal{M}$ and if $x,y \in \mathcal{M}$, then there exists a unique minimal geodesic joining $x$ to $y$. The rest of the paper, we assume that $\mathcal{M}$ is a Hadamard manifold. The following results are known and will be useful.

\begin{proposition}\label{proposition 1.1}(~\cite{Sak1})
Let $\mathcal{M}$ be a Hadamard manifold and $p \in \mathcal{M}$. Then $\exp_p : T_p\mathcal{M} \rightarrow \mathcal{M}$ is a diffeomorphism, and for any two points $p, q \in \mathcal{M}$ there exists a unique normalized geodesic joining $p$ to $q$, which is, in fact, a minimal geodesic.
\end{proposition}

This proposition yields that $\mathcal{M}$ is diffeomorphic to the Euclidean space $\mathbb{R}^n$. Thus, we see that $\mathcal{M}$ has the same topology and differential structure as $\mathbb{R}^n$. Moreover, Hadamard manifolds and Euclidean spaces have some similar geometrical properties, one of the most important proprieties is illustrated in the following proposition.

\begin{proposition}\label{proposition 1.2}(~\cite{Sak1})
Let $\Delta(p_1, p_2, p_3)$ be a geodesic triangle in a Hadamard manifold $\mathcal{M}$. For each $i = 1, 2, 3(\mod 3)$, let $\gamma_i : [0, l_i] \rightarrow \mathcal{M}$ denote the geodesic joining $p_i$ to $p_{i+1}$. Let $l_i = L(\gamma_i )$ and $\alpha_i:= \angle (\gamma'_i(0),-\gamma'_{i-1}(l_{i-1}))$ be the angle between tangent vectors $\gamma'_i(0)$ and $\gamma'_{i-1}(l_{i-1})$. Then
\item (i) $\alpha_1+\alpha_2+\alpha_3\leq \pi$;
\item (ii) $l^2_i+l^2_{i+1}-2l_il_{i+1}\cos \alpha_{i+1}\leq l^2_{i-1}$;
\item (iii) $l_{i+1} \cos \alpha_{i+2}+l_i\cos \alpha_i\geq l_{i+2}$.
\end{proposition}

Considering the distance and the exponential map, we have that the following inequalities are equivalent to Proposition \ref{proposition 1.2} (ii) and (iii).
\begin{equation}\label{eq0.5}
\begin{aligned}
d^2(p_i,p_{i+1})+d^2(p_{i+1},p_{i+2})-2\langle \exp^{-1}_{p_{i+1}}p_i,\exp^{-1}_{p_{i+1}}p_{i+2} \rangle \leq d^2(p_{i-1},p_i),
\end{aligned}
\end{equation}
and
\[
d^2(p_i , p_{i+1}) \leq \langle \exp^{-1}_{p_i}p_{i+2}, \exp^{-1}_{p_i}p_{i+1} \rangle + \langle \exp^{-1}_{p_{i+1}}p_{i+2}, \exp^{-1}_{p_{i+1}}p_i \rangle,
\]
since $\langle \exp^{-1}_{p_{i+1}}p_i, \exp^{-1}_{p_{i+1}}p_{i+2} \rangle = d(p_i,p_{i+1})d(p_{i+1}, p_{i+2}) \cos \alpha_{i+1} $. For further details, we refer to  \cite{Fer1}.

\begin{lemma}\label{lemma 2.2}(~\cite{Rei1})
Let $\Delta(p, q, r)$ be a geodesic triangle in Hadamard manifold $\mathcal{M}$. Then there exists a triangle $\Delta(\bar{p}, \bar{q}, \bar{r})$ ($\bar{p}, \bar{q}, \bar{r} \in \mathbb{R}^2$) for $\Delta(p, q, r)$ such that
\[
d(p,q)=\|\bar{p}-\bar{q}\|, \quad d(q,r)=\|\bar{q}-\bar{r}\|, \quad d(r,p)=\|\bar{r}-\bar{p}\|.
\]
\end{lemma}

The triangle $\Delta(\bar{p}, \bar{q}, \bar{r})$ is called the comparison triangle of the geodesic triangle $\Delta(p, q, r)$, which is unique up to isometry of $\mathcal{M}$.

\begin{lemma}\label{lemma 1.1}(~\cite{Li1})
Let $\{x_n\}$ be a sequence in $\mathcal{M}$ such that $x_n \rightarrow x_0 \in  \mathcal{M}$. Then the following assertions hold.
\item  (i)  For any $y \in \mathcal{M}$, we have $\exp^{-1}_{x_n}y\rightarrow \exp^{-1}_{x_0}y$ and  $\exp^{-1}_y x_n \rightarrow \exp^{-1}_y x_0$.
\item  (ii) If $v_n\in T_{x_n}\mathcal{M}$ and $v_n\rightarrow v_0$, then $v_0 \in T_{x_0}\mathcal{M}$.
\item  (iii) Given $u_n,v_n \in T_{x_n}\mathcal{M}$ and $u_0,v_0\in T_{x_0}\mathcal{M}$, if $u_n\rightarrow u_0$ and $v_n\rightarrow v_0$, then $\langle u_n,v_n \rangle \rightarrow \langle u_0,v_0 \rangle$.
\item  (iv) For any $u\in T_{x_0}\mathcal{M}$, the function $A: \mathcal{M} \rightarrow T\mathcal{M}$ defined by $A(x)=P_{x,x_0}u$ for each $x\in \mathcal{M}$ is continuous on $\mathcal{M}$.
\end{lemma}

\begin{definition}\label{definition 1.2}
A subset $C$ is said to be convex if, for every two points $x$ and $y$ in $C$, the geodesic joining $x$ to $y$ is contained in $C$,
that is, if $\gamma: [a, b] \rightarrow \mathcal{M} $ is a geodesic such that $x = \gamma(a)$ and $y = \gamma(b)$, then $\gamma((1-t)a+tb) \in C$ for all $t \in [0, 1]$.
\end{definition}

\begin{definition}\label{definition 1.21}
A real function $f$ defined in $\mathcal{M}$ is said to be  convex if, for any geodesic $\gamma$ of $\mathcal{M}$,
the composition function $f\circ \gamma : [a, b] \rightarrow \mathbb{R}$ is convex, that is,
\[
(f\circ \gamma)(ta+(1-t)b)\leq t(f\circ \gamma)(a)+(1-t)(f\circ \gamma)(b),
\]
\end{definition}
where $a, b \in \mathbb{R}$, and $t \in  [0, 1]$.

\begin{definition}\label{definition 1.12}
Let $f: \mathcal{M}  \rightarrow \mathbb{R}$ be a convex and $x\in \mathcal{M}$. A vector $p \in T_x\mathcal{M}$ is
said to be  a subgradient of $f$ at $x$ if for any $y\in \mathcal{M}$,
\[
f(y)\geq f(x)+\langle p,\exp^{-1}_x y \rangle.
\]
\end{definition}
The set of all subgradients of $f$, denoted by $\partial f(x)$, is called the subdifferential of $f$ at $x$, which is closed convex set.
Let $D(\partial f)$ denote the domain of $\partial f$ defined by $D(\partial f)=\{x\in \mathcal{M}: \partial f(x)\neq \emptyset\}$.
The existence of subgradients for convex functions is guaranteed by the following proposition.

\begin{proposition}\label{proposition 1.3}(~\cite{Fer1})
Let $\mathcal{M}$ be a Hadamard manifold and $f: [a, b] \rightarrow \mathbb{R}$ be convex. Then, for all $x \in \mathcal{M}$,
the subdifferential $\partial f(x)$ of $f$ at $x$ is nonempty. That is, $D(\partial f)= \mathcal{M}$.
\end{proposition}

Next, we recall some concepts of monotonicity of a bifunction.

\begin{definition}\label{definition 1.1}(~\cite{Nem2,Nem3})
A bifunction $f : C \times C \rightarrow \mathbb{R}$ is said to be
\item (i)  monotone if for any $(x, y) \in C \times C$,
\[
f(x, y) + f(y, x) \leq 0;
\]
\item (ii) strongly monotone if for any $(x, y) \in C \times C$, there exists a positive constant $\gamma$ such that
\[
f(x,y)+f(y,x)\leq -\gamma d^2(x,y);
\]
\item (iii) pseudomonotone if for any $(x, y) \in C \times C$,
\[
f(x,y)\geq 0\Rightarrow f(y,x)\leq 0;
\]
 \item (iv) strongly pseudomonotone if for any $(x, y) \in C \times C$, there exists a positive constant $\gamma$ such that
\[
f(x,y)\geq 0 \Rightarrow f(y,x) \leq -\gamma d^2(x,y).
\]
\end{definition}
It follows from the definitions that the following implications hold:
\[
(ii)\Rightarrow (i) \Rightarrow (iii) \quad \mbox{ and } \quad  (ii)\Rightarrow (iv) \Rightarrow (iii).
\]

\begin{definition}\label{definition 1.13}(~\cite{Mas2})
A bifunction $f : C \times C \rightarrow \mathbb{R}$ is said to satisfy a Lipschitz-type condition on $C$ if there exist two positive constants $\gamma_1$ and $\gamma_2$ such that
\[
f(x,y)+f(y,z)\geq f(x,z)-\gamma_1d^2(x,y)-\gamma_2d^2(y,z), \quad \forall  x, y, z \in C.
\]
\end{definition}

Let $f: \mathcal{M}\rightarrow \mathbb{R}$ be a convex, proper and lower semicontinuous function. The proximal point algorithm generates, for a initial point $x_0 \in \mathcal{M}$, a sequence $\{x_n\} \subset \mathcal{M}$, which is defined by the following:
\begin{equation}\label{eq0.6}
\begin{aligned}
x_{n+1}=\arg\min_{t \in \mathcal{M} } \left\{f(t)+\frac{\lambda_n}{2}d^2(x_n,t) \right\} , \quad \lambda_n \subset (0,+\infty).
\end{aligned}
\end{equation}
The following lemmas are useful for the convergence of our proposed algorithm.

\begin{lemma}\label{lemma 1.5}(~\cite{Fer1})
Let $f: \mathcal{M}\rightarrow \mathbb{R}$ be a convex, proper and lower semicontinuous function. Then the sequence $\{x_n\}$ generated by \eqref{eq0.6} is well defined, and characterized by
\[
\lambda_n(\exp^{-1}_{x_{n+1}}x_n) \in \partial f(x_{n+1}).
\]
\end{lemma}

\begin{definition}\label{definition 1.3}(~\cite{Fer1})
Let $X$ be a complete metric space and let $C \in X$ be a nonempty set.
A sequence $\{x_n\} \subset X$ is said to be  Fej$\acute{e}$r convergent to $C$ if, for all $ y \in C$ and $n\geq 0$, $d(x_{n+1},y) \leq d(x_n, y)$.
\end{definition}

\begin{lemma}\label{lemma 1.7}(~\cite{Li1})
Let $X$ be a complete metric space and let $C \in X$ be a nonempty set. Let $x_n \subset X$ be Fej$\acute{e}$r convergent to $C$ and suppose that any cluster point of $\{x_n\}$ belongs to C. Then $\{x_n\}$ converges to some point in  $C$.
\end{lemma}

\section{The explicit extragradient algorithm for the equilibrium problem}

In this section, we introduce an extragradient algorithm involving pseudomonotone for equilibrium problem \eqref{EP} on Hadamard manifolds.
Unlike existing extragradient-like methods for problem \eqref{EP}, the stepsizes used in the presented algorithm are independent of the Lipschitz-type constants.
From now, let $C$ be a nonempty closed convex set of $\mathcal{M}$. Next, let $f: C \times C \rightarrow \mathbb{R}$ be a bifunction.
In order to obtain the convergence of Algorithm \ref{algorithm 0.1}, we make the following hypothesizes regarding the bifunction:

\begin{enumerate}[(C1)]
\item  $f$ is pseudomontone on $C$ and $f(x,x)=0$ for all $x \in C$;\label{C1}
\item  $f$ satisfies the Lipschitz-type condition; \label{C2}
\item  $f(x, \cdot)$ is convex and lower semicontinuous on $C$ for all $x \in C$; \label{C3}
\item  $\limsup_{n\rightarrow \infty} f(x_n,y)\leq f(x,y)$ for each $y \in C$ and each $\{x_n\} \subset C$ with $x_n \rightharpoonup x$. \label{C4}
\end{enumerate}
For the sake of simplicity in the presentation, we will use the notation $[t]_+=\max\{0,t\}$
and adopt the conventions $\frac{0}{0}=+\infty$ and $\frac{a}{0}= +\infty$ $(a\neq 0)$.
More precisely, the algorithm is described as follows:

\begin{algorithm}[h]
\caption{(An explicit extragradient algorithm for pseudomonotone EPs)}

\begin{enumerate}[]\label{algorithm 0.1}
\item \textbf{Initialization: } Choose $x_0 \in  C$ and $\lambda_0 > 0, \mu \in (0, 1)$.
\item \textbf{Iterative Steps:} Given the current iterate $x_n\in C$ and $\lambda_n (n\geq 0)$, calculate $x_{n+1}, \lambda_{n+1}$ as follows.
\item   Compute
\[
\begin{aligned}
\left\{
     \begin{array}{ll}
        y_n=\arg\min_{y\in C} \left\{f(x_n,y)+\frac{1}{2\lambda_n}d^2(x_n,y) \right\} ,\\
        x_{n+1}=\arg\min_{y\in C} \left\{f(y_n,y)+\frac{1}{2\lambda_n}d^2(x_n,y) \right\},
        \end{array}
    \right.
\end{aligned}
\]
\item    and set
\[
 \lambda_{n+1}=\min \left\{ \lambda_n, \frac{\mu(d^2(x_n,y_n)+d^2(x_{n+1},y_n))}{2[f(x_n,x_{n+1})-f(x_n,y_n)-f(y_n,x_{n+1})]_+} \right\}.
\]
\item \textbf{Stopping Criterion:} If $y_n = x_n$, then stop and $x_n$ is the solution of equilibrium problem \eqref{EP}.
\end{enumerate}
\end{algorithm}

\begin{remark}\label{remark 1.0}
Under hypothesis (C\ref{C2}), we see that there exist  constants $\gamma_1 > 0, \gamma_2 > 0$ such that
\[
\begin{aligned}
f(x_n,x_{n+1})-f(x_n,y_n)-f(y_n,x_{n+1})&\leq \gamma_1 d^2(x_n,y_n)+\gamma_2 d^2(x_{n+1},y_n)\\
                                        &\leq \max \{\gamma_1,\gamma_2\} (d^2(x_n,y_n)+d^2(x_{n+1},y_n)).
\end{aligned}
\]
Thus, from the definition of $\{\lambda_n\}$, we see that this sequence is bounded from below
by $\left\{\lambda_0, \frac{\mu}{2\max \{\gamma_1,\gamma_2\}} \right\}$. Moreover,  $\{\lambda_n\}$ is non-increasing monotone. Therefore,
there exists a real number $\lambda > 0$ such that $ \lim_{n \rightarrow \infty} \lambda_n = \lambda $. In fact, from the definition of
$\{\lambda_{n+1}\}$, if $ f(x_n,x_{n+1})-f(x_n,y_n)-f(y_n,x_{n+1}) \leq 0$, then $\lambda_{n+1} := \lambda_n$.
\end{remark}

We are now turn to the main result regarding the convergence of the proposed algorithm.

\begin{theorem}\label{theorem 3.2}
Assume that  bifunction $f$ satisfies (C\ref{C1})-(C\ref{C4}).
Then, the sequences $\{x_n\}$ generated by Algorithm \ref{algorithm 0.1} converges to a solution of equilibrium problem \eqref{EP}.
\end{theorem}

\begin{proof}
From definition of $x_{n+1}$ and Lemma \ref{lemma 1.5}, one obtains
\begin{equation}\label{eq1.1}
\begin{aligned}
\langle \exp^{-1}_{x_{n+1}} x_n, \exp^{-1}_y x_{n+1} \rangle \geq \lambda_n f(y_n,x_{n+1})-\lambda_n f(y_n,y), \quad \forall y\in C.
\end{aligned}
\end{equation}
From the definition of $\lambda_{n+1}$, one concludes
\[
f(x_n,x_{n+1})-f(x_n,y_n)-f(y_n,x_{n+1})\leq \frac{\mu(d^2(x_n,y_n)+d^2(x_{n+1},y_n))}{2\lambda_{n+1}}.
\]
Since $\lambda_n >0$, one can write the above inequality as
\begin{equation}\label{eq1.2}
\begin{aligned}
\lambda_n f(y_n,x_{n+1})\geq \lambda_n(f(x_n,x_{n+1})-f(x_n,y_n))-\frac{\mu \lambda_n(d^2(x_n,y_n)+d^2(x_{n+1},y_n))}{2\lambda_{n+1}}.
\end{aligned}
\end{equation}
Combining \eqref{eq1.2} through \eqref{eq1.1} yields that
\begin{equation}\label{eq1.3}
\begin{aligned}
\langle \exp^{-1}_{x_{n+1}} x_n, \exp^{-1}_y x_{n+1} \rangle \geq &\lambda_n(f(x_n,x_{n+1})-f(x_n,y_n))-\frac{\mu \lambda_n}{2\lambda_{n+1}}d^2(x_n,y_n)\\
&-\frac{\mu \lambda_n}{2\lambda_{n+1}}d^2(x_{n+1},y_n)-\lambda_nf(y_n,y).
\end{aligned}
\end{equation}
It also follows from the definition of $y_n$ and Lemma \ref{lemma 1.5} that
\begin{equation}\label{eq1.4}
\begin{aligned}
\lambda_n (f(x_n,x_{n+1})-f(x_n,y_n)) \geq \langle \exp^{-1}_{y_n}x_n, \exp^{-1}_{y_n}{x_{n+1}}\rangle.
\end{aligned}
\end{equation}
From the relations \eqref{eq1.3} and \eqref{eq1.4}, one obtains
\begin{equation}\label{eq1.5}
\begin{aligned}
2\langle \exp^{-1}_{x_{n+1}} x_n, \exp^{-1}_y x_{n+1} \rangle \geq & 2\langle \exp^{-1}_{y_n}x_n, \exp^{-1}_{y_n}{x_{n+1}} \rangle-\frac{\mu \lambda_n}{\lambda_{n+1}}d^2(x_n,y_n)\\
&-\frac{\mu \lambda_n}{\lambda_{n+1}}d^2(x_{n+1},y_n)-2\lambda_nf(y_n,y).
\end{aligned}
\end{equation}
Let $\Delta(x_{n+1},x_n,y) \subseteq \mathcal{M}$ be the geodesic triangle and using \eqref{eq0.5}, it follows that
\begin{equation}\label{eq1.6}
\begin{aligned}
2\langle \exp^{-1}_{x_{n+1}} x_n, \exp^{-1}_{x_{n+1}} y \rangle \geq d^2(x_n,x_{n+1})+d^2(x_{n+1},y)-d^2(x_n,y).
\end{aligned}
\end{equation}
Similarly, let $\Delta(x_n,x_{n+1},y_n) \subseteq \mathcal{M}$ be the geodesic triangle and using \eqref{eq0.5}, then
\begin{equation}\label{eq1.7}
\begin{aligned}
2\langle \exp^{-1}_{y_n}x_n, \exp^{-1}_{y_n}{x_{n+1}} \rangle \geq d^2(y_n,x_n)+d^2(y_n,x_{n+1})-d^2(x_n,x_{n+1}).
\end{aligned}
\end{equation}
Combining the relations \eqref{eq1.5}–\eqref{eq1.7}, one arrives at
\begin{equation}\label{eq1.8}
\begin{aligned}
d^2(x_{n+1},y)\leq &d^2(x_n,y)-(1-\frac{\mu \lambda_n}{\lambda_{n+1}})d^2(x_n,y_n)\\
&-(1-\frac{\mu \lambda_n}{\lambda_{n+1}})d^2(x_{n+1},y_n)+2 \lambda_nf(y_n,y).
\end{aligned}
\end{equation}
Taking $p \in EP(f,C)$, we have that $f(p,y_n)\geq 0$. It follows from the pseudomonotonicity of $f$ that $f(y_n,p)\leq 0$. Then, using $y=p \in C$ in \eqref{eq1.8}, we get
\begin{equation}\label{eq1.9}
\begin{aligned}
d^2(x_{n+1},p)\leq &d^2(x_n,p)-(1-\frac{\mu \lambda_n}{\lambda_{n+1}})d^2(x_n,y_n)\\
&-(1-\frac{\mu \lambda_n}{\lambda_{n+1}})d^2(x_{n+1},y_n).
\end{aligned}
\end{equation}
Let $\kappa$ be fixed in $(0,1-\mu)$. Since $\lim_{n\rightarrow \infty}\lambda_n=\lambda >0$, one asserts that
\[
\lim_{n \rightarrow \infty}(1-\frac{\mu \lambda_n}{\lambda_{n+1}})=1-\mu > \kappa >0.
\]
Thus, there exists $n_0\geq 0$ such that, for all $n \geq n_0$,
\begin{equation}\label{eq1.10}
\begin{aligned}
1-\frac{\mu\lambda_n}{\lambda_{n+1}}> \kappa >0.
\end{aligned}
\end{equation}
Adding \eqref{eq1.10} into \eqref{eq1.9}, one obtains
\[
d^2(x_{n+1},p)\leq d^2(x_n,p)-\kappa(d^2(x_n,y_n)+d^2(x_{n+1},y_n)),
\]
which implies that
\begin{equation}\label{eq1.11}
\begin{aligned}
a_{n+1}\leq a_n-b_n,
\end{aligned}
\end{equation}
where
\[
\begin{aligned}
a_n=d^2(x_n,p) \mbox{ and } b_n=\kappa(d^2(x_n,y_n)+d^2(x_{n+1},y_n)).
\end{aligned}
\]
It is obvious that $\lim_{n \rightarrow \infty} a_n$ exists, and $\lim_{n \rightarrow \infty} b_n=0 $. Hence $\{x_n\}$ is bounded.
Thus, we conclude from the definition of $b_n$ that
\begin{equation}\label{eq1.12}
\begin{aligned}
\lim_{n\rightarrow \infty}d^2(x_n,y_n)=\lim_{n\rightarrow \infty}d^2(x_{n+1},y_n)=0,
\end{aligned}
\end{equation}
which implies from the boundedness of $\{x_n\}$ that $\{y_n\}$ is bounded. Using \eqref{eq0.5}, we obtain
\[
d^2(x_n,x_{n+1})\leq 2\langle \exp^{-1}_{x_{n+1}}x_n,\exp^{-1}_{x_{n+1}}y_n \rangle-d^2(x_{n+1},y_n)+d^2(x_n,y_n).
\]
We also have
\begin{equation}\label{eq1.13}
\begin{aligned}
\lim_{n\rightarrow \infty}d^2(x_n,x_{n+1})=0.
\end{aligned}
\end{equation}
We next prove that each weak cluster point of $\{x_n\}$ is in $EP(f, C)$.
We show that $\{x_n\}$ is bounded. Therefore there exists a subsequence $\{x_{n_k}\}$ of $\{x_n\}$ and $x^*\in C $ such that $x^*$ is a weak cluster point of $\{x_n\}$, i.e., $x_{n_k}\rightharpoonup x^*$. Hence, by using \eqref{eq1.12}, we have that  $y_{n_k}\rightharpoonup x^*$.
Replacing $n$ by $n_k$ in \eqref{eq1.8}, and taking $\limsup$ and using hypothesis (C\ref{C4}), we have
\begin{equation}\label{eq1.14}
\begin{aligned}
f(x^*,y)\geq\limsup_{k\rightarrow \infty}f(y_{n_k},y)\geq\frac{1}{2\lambda}\limsup_{k\rightarrow \infty} (d^2(x_{n_k+1},y)-d^2(x_{n_k},y)),\quad  \forall y\in C.
\end{aligned}
\end{equation}
On the other hand, from \eqref{eq0.5}, we obtain
\[
d^2(x_{n_k+1},y)-d^2(x_{n_k},y)\leq \langle \exp^{-1}_{x_{n_k+1}}x_{n_k}, \exp^{-1}_{x_{n_k+1}}y \rangle-d^2(x_{n_k+1},x_{n_k}).
\]
This together with \eqref{eq1.13} implies that
\begin{equation}\label{eq1.15}
\begin{aligned}
\lim_{k\rightarrow \infty}(d^2(x_{n_k+1},y)-d^2(x_{n_k},y))=0.
\end{aligned}
\end{equation}
Combining \eqref{eq1.14} and \eqref{eq1.15}, we get $f(x^*,y)\geq\limsup_{k\rightarrow \infty}f(y_{n_k},y)\geq 0$, for all $y\in C$. Therefore, $x^*~ \in ~EP(f, C)$.
From \eqref{eq1.9}, \eqref{eq1.10} and Definition \ref{definition 1.3}, we know that $\{x_n\}$ is Fej$\acute{e}$r convergent to $C$. Finally, Lemma \ref{lemma 1.7} implies that   sequence $\{x_n\}$ converges to a point of $EP(f, C)$. This completes the proof.
\qed
\end{proof}

\section{The R-linear rate of the convergence}

Algorithms in \cite{Kha1} have some special advantages that they are done without the prior knowledge of the Lipschitz-type constants of the bifunction.
However, in the case that  bifunction $f$ is strongly pseudomonotone (SP), the linear rate of convergence cannot be obtained for these algorithms.
In this section, we will establish the $R$-linear rate of the convergence of Algorithms \ref{algorithm 0.1}  under   hypothesis (SP) and (C\ref{C1})-(C\ref{C4}).
Under these assumptions, equilibrium problem \eqref{EP} has the unique solution, denoted by $\bar{x}$.
The rate of the convergence of the proposed algorithm is ensured by the following theorem.

\begin{theorem}\label{theorem 3.3}
Under   hypotheses (C\ref{C1})-(C\ref{C4}) and (SP), the sequence $\{x_n\}$ generated by
Algorithm \ref{algorithm 0.1} converges $R$-linearly to the unique solution $\bar{x}$ of equilibrium problem \eqref{EP}.
\end{theorem}

\begin{proof}
Using the \eqref{eq1.8} with $y = \bar{x}$, we obtain
\begin{equation}\label{eq3.1}
\begin{aligned}
d^2(x_{n+1},\bar{x})\leq &d^2(x_n,\bar{x})-(1-\frac{\mu \lambda_n}{\lambda_{n+1}})d^2(x_n,y_n)\\
&-(1-\frac{\mu \lambda_n}{\lambda_{n+1}})d^2(x_{n+1},y_n)+2 \lambda_nf(y_n,\bar{x}).
\end{aligned}
\end{equation}
Since $\bar{x} \in  EP(f, C)$ we have  $f(\bar{x}, y_n) \geq 0$. From assumption (SP), we get that
\begin{equation}\label{eq3.2}
\begin{aligned}
f (y_n, \bar{x}) \leq -\rho d^2(y_n, \bar{x}),
\end{aligned}
\end{equation}
where $\rho$ is some positive real number.
Adding \eqref{eq3.2} into \eqref{eq3.1}, we have
\begin{equation}\label{eq3.3}
\begin{aligned}
d^2(x_{n+1},\bar{x})\leq &d^2(x_n,\bar{x})-(1-\frac{\mu \lambda_n}{\lambda_{n+1}})d^2(x_n,y_n)\\
&-(1-\frac{\mu \lambda_n}{\lambda_{n+1}})d^2(x_{n+1},y_n)-2 \rho \lambda_n d^2(y_n,\bar{x}).
\end{aligned}
\end{equation}
Since $\{\lambda_n\}$ is non-increasing monotone and $\lim_{n\rightarrow \infty}\lambda_n=\lambda >0$, one has that
 $\lambda_n \geq \lambda_\infty =\lambda$ for all $n \geq 0$. Then, it follows  from \eqref{eq3.3} that
\begin{equation}\label{eq3.4}
\begin{aligned}
d^2(x_{n+1},\bar{x})\leq &d^2(x_n,\bar{x})-(1-\frac{\mu \lambda_n}{\lambda_{n+1}})d^2(x_n,y_n)\\
&-(1-\frac{\mu \lambda_n}{\lambda_{n+1}})d^2(x_{n+1},y_n)-2 \rho \lambda d^2(y_n,\bar{x}).
\end{aligned}
\end{equation}
Letting $\kappa$ be fixed in $(0,\frac{1-\mu}{2})$, we find that
\[
\lim_{n \rightarrow \infty}(1-\frac{ \mu \lambda_n}{\lambda_{n+1}})=1-\mu > 2\kappa >0.
\]
Thus, there exists $n_0\geq 0$ such that, for all $n \geq n_0$,
\begin{equation}\label{eq3.5}
\begin{aligned}
1-\frac{ \mu \lambda_n}{\lambda_{n+1}}> 2 \kappa >0.
\end{aligned}
\end{equation}
It follows from \eqref{eq3.4} and \eqref{eq3.5} that, for all $n \geq n_0$,
\begin{equation}\label{eq3.6}
\begin{aligned}
d^2(x_{n+1},\bar{x}) &\leq d^2(x_n,\bar{x})-(1-\frac{\mu \lambda_n}{\lambda_{n+1}})d^2(x_n,y_n)-2 \rho \lambda d^2(y_n,\bar{x})\\
&\leq d^2(x_n,\bar{x})-2\kappa d^2(x_n,y_n)-2 \rho \lambda d^2(y_n,\bar{x})\\
&\leq d^2(x_n,\bar{x})-\min\{ \kappa, \rho \lambda\} \left\{2d^2(x_n,y_n)+2d^2(y_n,\bar{x})\right\}\\
&\leq d^2(x_n,\bar{x})-\min\{ \kappa, \rho \lambda\} d^2(x_n,\bar{x})\\
&=r d^2(x_n, \bar{x}),
\end{aligned}
\end{equation}
where $r =1-\min\{ \kappa, \rho \lambda\}\in (0,1)$. In view of \eqref{eq3.6}, one concludes that
\[
d^2(x_{n+1},\bar{x}) \leq r^{n-n_0+1}d^2(x_{n_0},\bar{x}), \quad \forall n \geq n_0,
\]
or $d^2(x_{n+1},\bar{x}) \leq M r^n$ for all $n \geq n_0$, where $M= r^{1-n_0}d^2(x_{n_0},\bar{x})$. This finishes the proof.
\qed
\end{proof}

\section{Numerical experiment}
In this section, we illustrate the convergence behavior of our proposed Algorithm \ref{algorithm 0.1} through an equilibrium problem \eqref{EP}, which is relative to a strongly pseudomonotone bifunction. We use the \emph{fmincon} function in the MATLAB Optimization toolbox to solve the optimization problem. All the programs are executed in MATLAB2018a on a PC Desktop Intel(R) Core(TM) i5-8250U CPU @ 1.60GHz 1.800 GHz, RAM 8.00 GB. MATLAB codes to reproduce the experiments are freely available at \url{https://github.com/bingtan72/Fan2020EGM4EPonHM}.

\begin{example}\label{ex1}
Form \cite[Example 1]{Ansari2019}, let $\mathbb{R}^{++}=\{x \in \mathbb{R}: x>0\}$ and $\mathcal{M}=\left(\mathbb{R}^{++},\langle\cdot, \cdot\rangle\right)$ be the Riemanian manifold with the metric $\langle m, n\rangle:=m n$. Thus, the sectional curvature of $ \mathcal{M} $ is $ 0 $. $T_{x} \mathcal{M}$ denotes the tangent space at  $x \in \mathcal{M}$, equals $\mathbb{R}$. The Riemannian distance $d: \mathcal{M} \times \mathcal{M} \rightarrow \mathbb{R}^{+}$ is defined by
$d(x, y):=|\ln (x / y)|$. Then $ \mathcal{M} $ is a Hadamard manifold. Let $\gamma:[0,1] \rightarrow \mathcal{M}$  be a geodesic starting from $x=\gamma(0)$ with velocity $v=\gamma^{\prime}(0) \in T_{x} \mathcal{M}$ defined by $ \gamma(t):=x e^{(v /x)t} $.
Hence, we get that $ \exp _{x} t v=x e^{(v /x)t} $.
For any $x, y \in \mathcal{M}$, we obtain
\[
y=\exp _{x}\left(d(x, y) \frac{\exp _{x}^{-1} y}{d(x, y)}\right)=x e^{\left(\frac{\exp _{x}^{-1} y}{x d(x, y)}\right) d(x, y)}=x e^{\frac{\exp _{x}^{-1}y}{x}},
\]
and thus, the inverse of exponential map is
$ \exp _{x}^{-1} y=x \ln \left(y / {x}\right) $.

Next, we consider an extension of a Nash-Cournot oligopolistic equilibrium model \cite{Nash} with the price function and fee-fax function being affine. Assume that there are $ n $ companies. Let $ x=(x_{1},x_{2},\ldots, x_{n}) $ be a vector, and its elements $ x_{i} $ represent the number of  goods produced by company $ i $.  We suppose that the price function $p_{i}(s)$ is a decreasing affine of  $s=\sum_{i=1}^{n} x_{i}$ such as $p_{i}(s)=a_{i}-b_{i} s$, where $a_{i}, b_{i} \geq 0$. Then the profit of the company  $ i $ is given by $f_{i}(x)=p_{i}(s) x_{i}-c_{i}\left(x_{i}\right)$, where $c_{i}\left(x_{i}\right)$ is the tax and fee for generating $ x_{i} $. Set $\Phi_{i}=\left[x_{i, \min }, x_{i, \max }\right]$ is the strategy set of the company $ i $.  Therefore, $\Phi=\Phi_{1} \times \cdots \times \Phi_{n}$ is the strategy set of the model. In fact, each company $ i $ tries to maximize its own profits by choosing the corresponding production level $ x_{i} $. The common method of this model is based on the well-known Nash equilibrium concept.

We recall that a point $x^{*}=\left(x_{1}^{*}, x_{2}^{*}, \ldots, x_{n}^{*}\right) \in \Phi=\Phi_{1} \times \cdots \times \Phi_{n}$ is called an equilibrium point of the model if
\[
f_{i}\left(x^{*}\right) \geq f_{i}\left(x^{*}\left[x_{i}\right]\right), \quad \forall x_{i} \in \Phi_{i}, \forall i=1, \ldots, n,
\]
where  $x^{*}\left[x_{i}\right]$ stands for the vector obtained from $x^{*}$ by replacing $x_{i}^{*}$ with $x_{i}$. Set  $f(x, y)=\phi(x, y)-\phi(x, x)$, where $\phi(x, y)=-\sum_{i=1}^{n} f_{i}\left(x\left[y_{i}\right]\right)$. The problem of finding a Nash equilibrium point of the model can be expressed as:
\[
\text { Find } x^{*} \in \Phi, \quad \text { such that } f\left(x^{*}, x\right) \geq 0, \quad \forall x \in \Phi.
\]
We suppose that the tax-fee function $c_{i}\left(x_{i}\right)$  is increasing and affine for every $ i $. This assumption means that as the number of products increases, the taxes and expenses for producing a unit increase. Here, the bifunction $ f $ can be expressed as $ f(x, y)=\langle C x+D y+q, y-x\rangle $, where $q \in \mathbb{R}^{n}$ and $C, D$  are two matrices of order $ n $ such that $ D $ is symmetric positive semidefinite and $D-C$ is symmetric negative semidefinite. We consider here that $D-C$ is symmetric negative definite. From the property of $D-C$, if $f(x, y) \geq 0$, we have
\[\begin{aligned}
f(y, x) & \leq f(y, x)+f(x, y) \\
&=\langle C y+D x+q, x-y\rangle+\langle C x+D y+q, y-x\rangle \\
&=\langle(C-D) y+(D-C) x, x-y\rangle \\
&=(x-y)^{\mathrm{T}}(D-C)(x-y) \\
& \leq-\delta d^{2}(x, y),
\end{aligned}\]
where $\delta>0$. Then $ f $ is strongly pseudomonotone, i.e., assumption (C1) holds for $ f $. Furthermore, it is easy to prove that $ f $ satisfies the Lipschitz-type condition, see, e.g., \cite{Hieu2016}, (C\ref{C2}) is fulfilled. Assumption (C\ref{C3}) and (C\ref{C4}) are automatically fulfilled. Hence, Algorithm \ref{algorithm 0.1}  can be applied in this case.

For the numerical experiment, we consider four companies, that are defined as follows:
\begin{table}[htbp]
	\centering
	\caption{Parameter settings for Example \ref{ex1}}
	\begin{tabular}{cllll}
		\toprule
		\multicolumn{1}{c}{Company $ i $} & \multicolumn{1}{c}{Price $ p_{i}(s) $} & \multicolumn{1}{c}{Tax $ c_{i}(x_{i}) $} & \multicolumn{1}{c}{Strategy set $\Phi_{i}$} &  \\
		\midrule
		1     & $p_{1}(s)=100-0.01 s$     &  $c_{1}\left(x_{1}\right)=20 x_{1}$     & $\Phi_{1}=[1000,2000]$       \\
		2     & $p_{2}(s)=110-0.02 s$      & $c_{2}\left(x_{2}\right)=15 x_{2}+100$      & $\Phi_{2}=[500,2500]$        \\
		3     & $p_{3}(s)=100-0.015 s$      &  $c_{3}\left(x_{3}\right)=17 x_{3}$     & $\Phi_{3}=[800,1500]$        \\
		4     &  $p_{4}(s)=115-0.05 s$     &  $c_{4}\left(x_{4}\right)=20 x_{4}+75$     & $\Phi_{4}=[500,3000]$        \\
		\bottomrule
	\end{tabular}%
	\label{tab:addlabel}%
\end{table}%

In our Algorithm \ref{algorithm 0.1}, The starting point is $x_{0}=(1000, 500, 800, 500)^{\mathrm{T}} \in \mathbb{R}^{4}$. In view of Algorithm \ref{algorithm 0.1}, we see that $y_{n}=x_{n}$, then $x_{n}$  is the solution of problem \eqref{EP}. Therefore, we use the sequence $\varepsilon_{n}=d\left(x_{n}, y_{n}\right)$ to study the convergence of the Algorithm \ref{algorithm 0.1}. The convergence of $\{\varepsilon_{n}\}$ to zero implies that sequence $ \{x_{n}\} $ converges to the solution of the problem.  Next, we show the behavior of $\{\varepsilon_{n}\}$ in Algorithm \ref{algorithm 0.1} for different initial $ \lambda_{0} $ and $ \mu $. We perform experiments for both number of iterations (\# iteration) and elapsed execution
time (Elapsed time [sec]). The numerical results are reported in Figs. \ref{fig_iter1}--\ref{fig_time2}. In these figures, the $ x $-axis represents the number of iterations or execution time, and the $ y $-axis represents the value of $\{\varepsilon_{n}\}$.
\begin{figure}[htbp]
\centering
\includegraphics[scale=0.4]{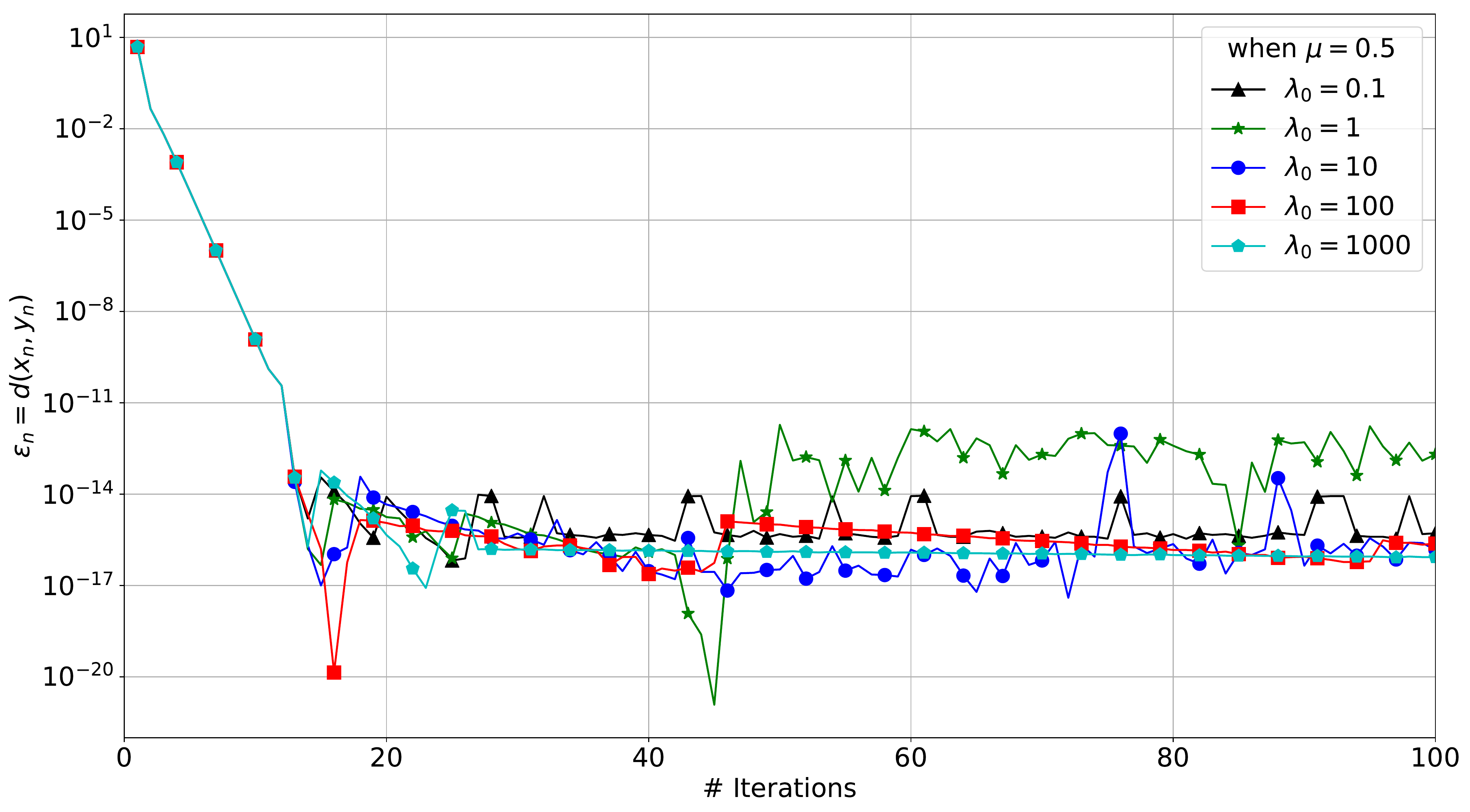}
\caption{Numerical behavior of $ \{\varepsilon_{n}\} $ in  Algorithm \ref{algorithm 0.1}  with the number of iterations}
\label{fig_iter1}
\end{figure}
\begin{figure}[htbp]
	\centering
	\includegraphics[scale=0.4]{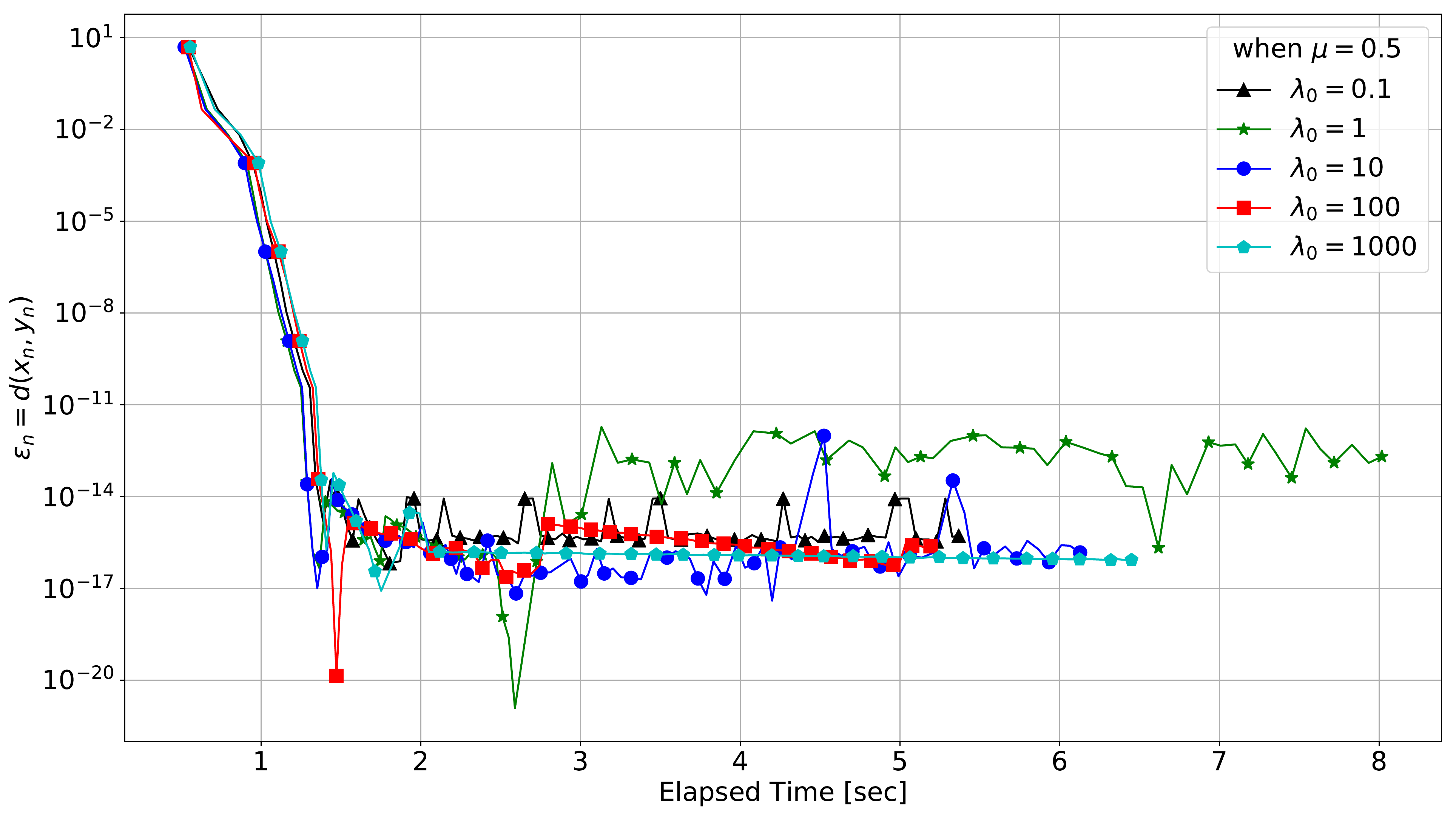}
	\caption{Numerical behavior of $ \{\varepsilon_{n}\} $ in  Algorithm \ref{algorithm 0.1}  with elapsed time}
	\label{fig_time1}
\end{figure}
\begin{figure}[htbp]
	\centering
	\includegraphics[scale=0.4]{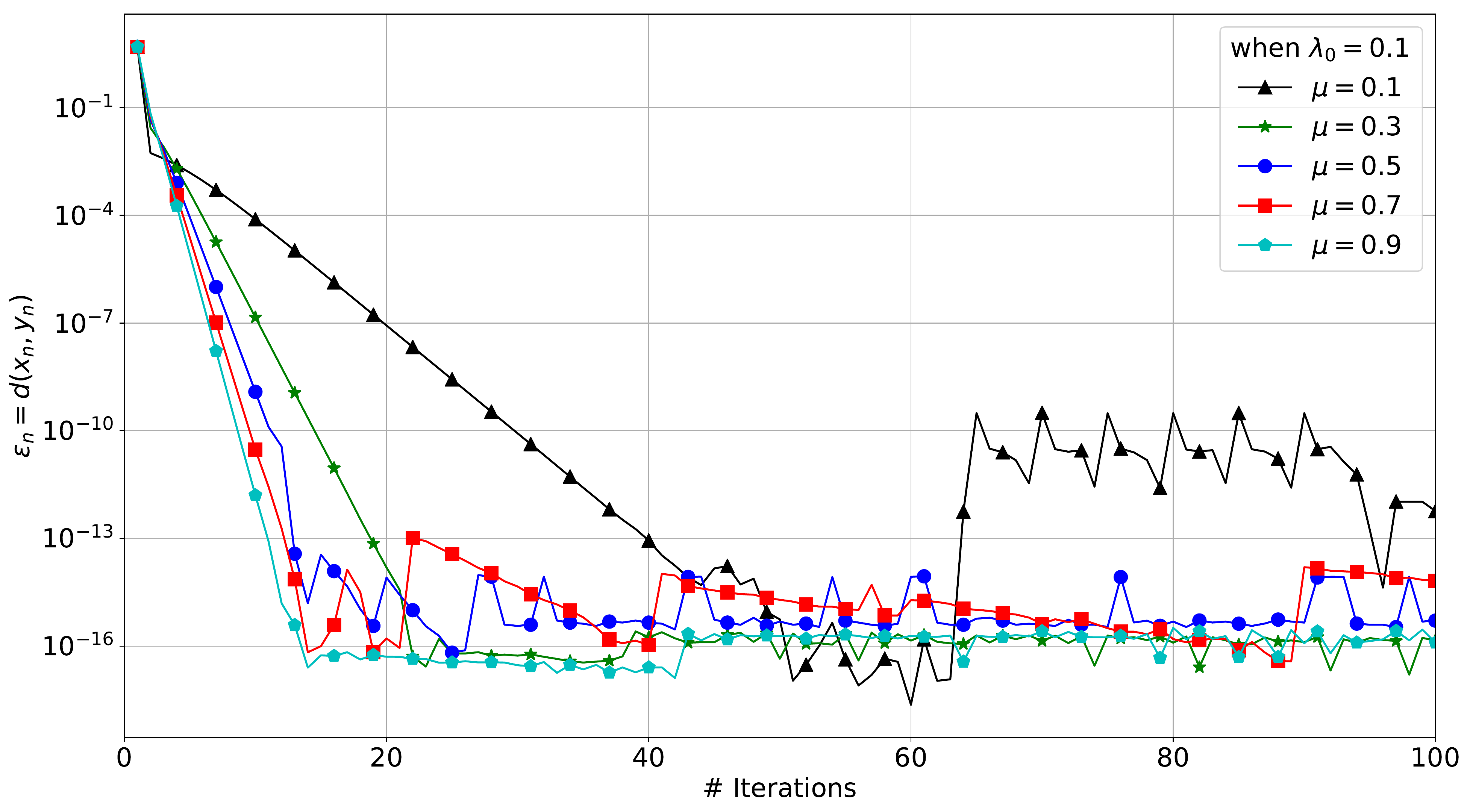}
	\caption{Numerical behavior of $ \{\varepsilon_{n}\} $ in  Algorithm \ref{algorithm 0.1}  with the number of iterations}
	\label{fig_iter2}
\end{figure}
\begin{figure}[htbp]
	\centering
	\includegraphics[scale=0.4]{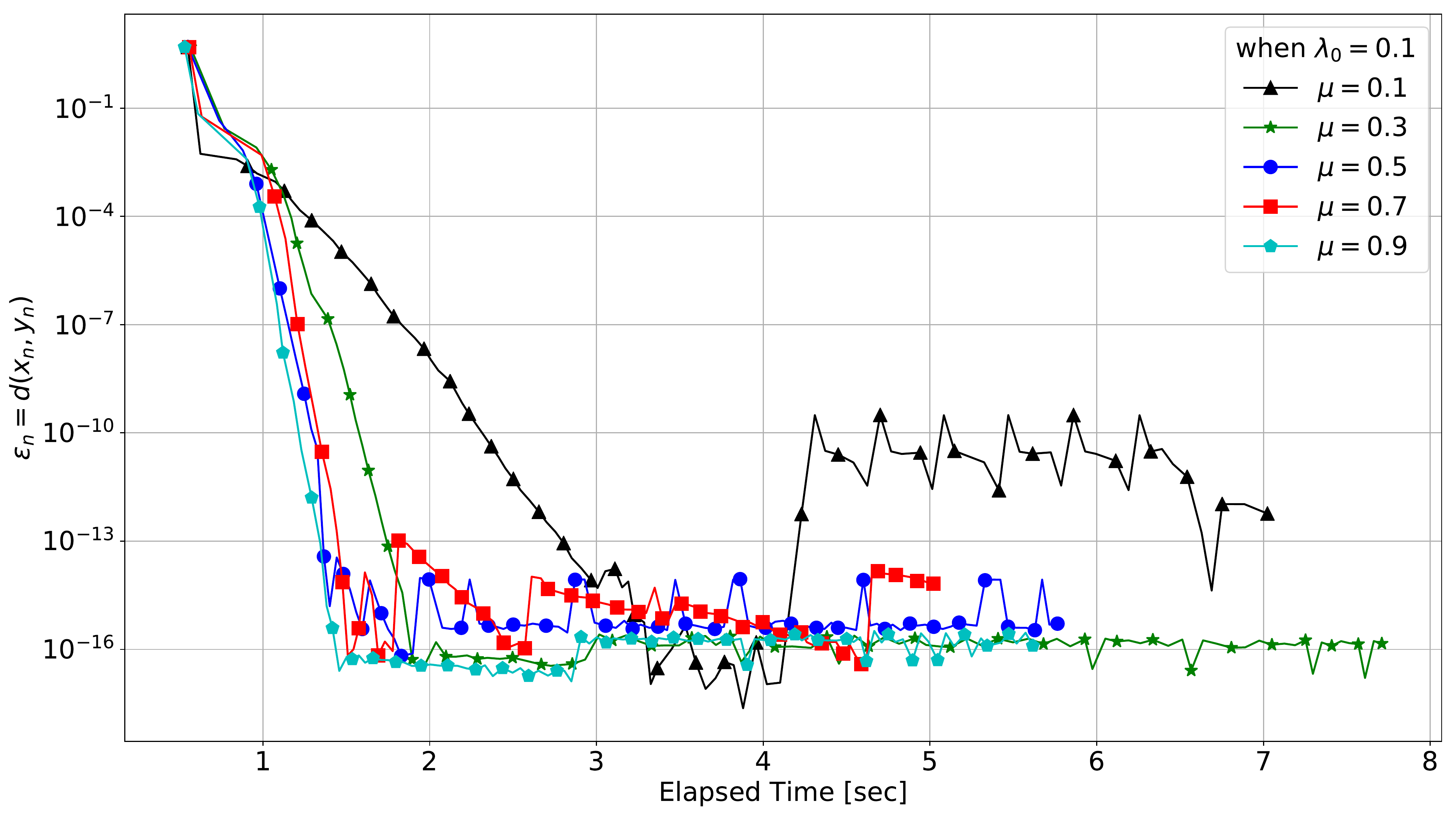}
	\caption{Numerical behavior of $ \{\varepsilon_{n}\} $ in  Algorithm \ref{algorithm 0.1}  with elapsed time}
	\label{fig_time2}
\end{figure}
\end{example}

From Figs. \ref{fig_iter1}--\ref{fig_time2}, we know that the rate of convergence of the sequence $ \{\varepsilon_{n}\} $ generated by Algorithm \ref{algorithm 0.1} is independent of parameters $ \lambda_{0} $ and $ \mu $. In addition, the first $ 20 $ iterations of $ \{\varepsilon_{n}\} $ are very fast, as the number of iterations increases, it seems to become unstable.

\section{Concluding remarks}

In this paper, we investigated the convergence of the new extragradient algorithm for the equilibrium problem involving pseudomonotone and Lipschitz-type bifunctions on Hadamard manifolds. A new stepsize rule allows us not to previously know  the information of the Lipschitz-type constants of bifunctions.
The convergence as well as the $ R $-linear rate of convergence of the algorithm were constructed.  The numerical behaviour of the extragradient algorithm was also discussed.
In order to devise more effective algorithms for problem \eqref{EP} on Hadamard manifolds, we will consider the geometric structure of manifolds in the future. It is of interest to do some numerical experiments and comparisons with other algorithms for practical problems on Riemannian manifolds.


\begin{thebibliography}{99}

\bibitem{Fan1}K. Fan. \textit{A minimax inequality and applications}, In: Shisha O, editor. Inequality III; New York: Academic Press.
1972; p. 103-113.

\bibitem{Blu1}E. Blum, W. Oettli. \textit{From optimization and variational inequalities to equilibrium problems},
Math. Student. 1994; \textbf{63}: 123-145.





\bibitem{Kon1}I.V. Konnov. \textit{Equilibrium Models and Variational Inequalities}, Elsevier, Amsterdam. 2007.

\bibitem{Fac1}F. Facchinei, J.S. Pang. \textit{Finite-Dimensional Variational Inequalities and Complementarity Problems}, Springer, Berlin. 2002.

\bibitem{Ius2}
B. Jadamba, A.A. Khan, F. Raciti.
\textit{Regularization of stochastic variational inequalities and a comparison of an  and a sample-path approach}, Nonlinear Anal. 2014; \textbf{94}: 65-83.

\bibitem{Chen1}J. Chen, Y.C. Liou, Z. Wan, J.C. Yao. \textit{A proximal point method for a class of monotone equilibrium problems with linear constraints}, Oper. Res. 2015; \textbf{15}(2): 275-288.

\bibitem{Qin2}N.T. Vinh, A. Gibali. \textit{Gradient projection-type algorithms for solving equilibrium problems and its applications},
Comput. Appl. Math. 2019; \textbf{38}: 119.


\bibitem{Bur1}R.S. Burachik, C.Y. Kaya, M. Mammadov.  \textit{An inexact modified subgradient algorithm for nonconvex optimization},
Comput. Optim. Appl. 2008; \textbf{45}: 1-24.


\bibitem{Mas1}G. Mastroeni.  \textit{Gap functions for equilibrium problems}, J. Glob. Optim. 2003; \textbf{27}: 411-426.


\bibitem{Hie2}D.V. Hieu, P.K. Quy, L.V. Vy. \textit{Explicit iterative algorithms for solving equilibrium problems}, Calcolo. 2019; \textbf{56}(2): 11.

\bibitem{Bac1}M. Ba$\check{c}\acute{a}$k, R. Bergmann, G. Steidl, A. Weinmann.  \textit{A second order nonsmooth variational model for restoring manifold-valued images},
SIAM J. Sci. Comput. 2016; \textbf{38}(1): A567–A597.

\bibitem{Ber2}R. Bergmann, J. Persch, G. Steidl. \textit{A parallel Douglas-Rachford algorithm for minimizing ROF-like functionals on images with values in symmetric Hadamard manifolds},
SIAM J. Imaging Sci. 2016; \textbf{9}(3): 901–937.




\bibitem{Lib1}X.B. Li, N.J. Huang, Q.H. Ansari, J.C. Yao.  \textit{Convergence rate of descent method with new inexact line-search on Riemannian manifolds},
J. Optim. Theory Appl. 2019; \textbf{180}(3): 830–854.



\bibitem{Ded3}J.P. Dedieu, P. Priouret, G. Malajovich.  \textit{Newton's method on Riemannian manifolds: covariant alpha theory},
IMA J. Numer. Anal. 2003; \textbf{23}(3): 395–419.


\bibitem{Li4}C. Li, J.H. Wang.  \textit{Newton's method on Riemannian manifolds: Smale's point estimate theory under the condition},
IMA J. Numer. Anal. 2006; \textbf{26}(2): 228–251.


\bibitem{ansari5}Q.H. Ansari, M. Islam, J.C. Yao, \textit{Nonsmooth variational inequalities on Hadamard manifolds}, Appl. Anal. 2020; \textbf{99}(2):  340-358.

\bibitem{Col1}V. Colao, G. L$\acute{o}$pez, G. Marino, V. Mart$\acute{i}$n-M$\acute{a}$rquez,  \textit{Equilibrium problems in Hadamard manifolds},
J. Math. Anal. Appl. 2012; \textbf{388}: 61-77.




\bibitem{Kha1}K. Khammahawong, P. Kumam, P. Chaipunya, et al. \textit{An extragradient algorithm for strongly pseudomonotone equilibrium problems on Hadamard manifolds}, Thai J. Math. 2020; \textbf{18}(1): 350-371.

\bibitem{Fer3}O.P. Ferreira,  L.R. Lucambio P$\acute{e}$rez, S.Z. N$\acute{e}$meth.
\textit{Singularities of monotone vector fields and an extragradient-type algorithm}, J. Global Optim. 2005; \textbf{31}(1): 133–151.

\bibitem{Lis1}S.L. Li, C. Li, Y.C. Liou, J.C. Yao. \textit{Existence of solutions for variational inequalities on Riemannian manifolds},
Nonlinear Anal. 2009; \textbf{71}(11): 5695–5706.

\bibitem{Led1}Y.S. Ledyaev, Q.J. Zhu. \textit{Nonsmooth analysis on smooth manifolds}, Trans. Amer. Math. Soc.
2007; \textbf{359}(8): 3687–3732.

\bibitem{Sak1}T. Sakai.  \textit{Riemannian Geometry, vol. 149 of Translations of Mathematical Monographs},
Amer. Math. Soc. Providence, RI. 1996.

\bibitem{Rei1}S. Reich.  \textit{Strong convergence theorems for resolvents of accretive operators in Banach spaces}, J. Math. Anal. Appl. 1980; \textbf{75}(1): 287-292.

\bibitem{Li1}C. Li, G. L$\acute{o}$pez, V.M. M$\acute{a}$rquez.  \textit{Monotone vector fields and the proximal point algorithm on Hadamard manifolds},
J. Lond. Math. Soc. 2009; \textbf{79}(3): 663-683.

\bibitem{Fer1}O.P. Ferreira, P.R. Oliveira.  \textit{Proximal point algorithm on Riemannian manifolds}, Optim. 2002; \textbf{51}(2): 257-270.

\bibitem{Nem2}S.Z. N$\acute{e}$meth.  \textit{Five kinds of monotone vector fields}, Pure Math. Appl. 1998; \textbf{9}(3): 417–428.

\bibitem{Nem3}S.Z. N$\acute{e}$meth.  \textit{Monotone vector fields}, Publ. Math. Debrecen. 1999; \textbf{54}(3): 437–449.

\bibitem{Mas2}G. Mastroeni. \textit{On Auxiliary Principle for Equilibrium Problems}, Equilibrium Problems and Variational Models.
Springer, Boston, MA. 2003; 289-298.

\bibitem{Ansari2019}
Q.H. Ansari, F. Babu, J.C. Yao. \textit{Regularization of proximal point algorithms in
Hadamard manifolds}, J. Fixed Point Theory Appl. 2019; \textbf{21}(25): 1–23.

\bibitem{Nash}
F. Facchinei, J.S. Pang. \textit{Finite-dimensional Variational Inequalities and Complementarity Problems}, Springer-Verlag, New York, 2007.	

\bibitem{Hieu2016}
D.V. Hieu. \textit{Parallel extragradient-proximal methods for split equilibrium problems}, Math. Model. Anal. 2016; 21: 478–501.


\end{thebibliography}
\end{document}